\documentclass[12pt]{amsart}
\usepackage[margin=1in]{geometry}
\usepackage{comment}
\usepackage{mathrsfs}
\usepackage{latexsym}
\usepackage{float}
\usepackage{amssymb}
\usepackage[all]{xy}
\usepackage{epsfig}
\usepackage{color}
\usepackage{graphics}
\usepackage{hyperref}

\newtheorem{Thm}[equation]{Theorem}
\newtheorem{Lem}[equation]{Lemma}
\newtheorem{Cor}[equation]{Corollary}

\begin{document}
\baselineskip=12pt

\title{Normal forms for rational 3-tangles}
\author{Bo-hyun Kwon, Jung Hoon Lee}
\date{Thursday, March 9,  2023}
\maketitle
\begin{abstract} 
 In this paper, we define the \textit{normal form} of collections of disjoint  three \textit{bridge arcs}  for  a given rational $3$-tangle. We show that  there is a sequence of \textit{normal jump moves} which leads one to the other for two normal forms of the same rational 3-tangle. 
\end{abstract}

\section{Introduction}

Originally, rational tangles introduced by John Conway (1970,~\cite{0}) meant  rational $2$-tangles. J. Conway completed the classification of rational tangles by assigning each rational tangle to the corresponding rational number. Moreover, if two rational tangles are equivalent then their corresponding rational numbers are the same.  Based on the definition of rational tangles, we can extend the meaning of the rational tangles to define rational $n$-tangles, $n\geq 2$.  For a long time after the rational tangle is defined  there are many attempts to classify rational $3$-tangles in advance.
As one of the trials to classify them, the first author~\cite{1} gave an algorithm to compare two rational $3$-tangles. However, it does not give a certain representative of each rational $3$-tangle. In this paper, we define the \textit{normal form} which is a special collection of disjoint three \textit{bridge arcs} of a given rational $3$-tangle $T$. We provide a clue to give a certain representative of each rational $3$-tangle.

\subsection{Bridge disk and bridge arc.}  
 \begin{figure}[htb]
\includegraphics[scale=.4]{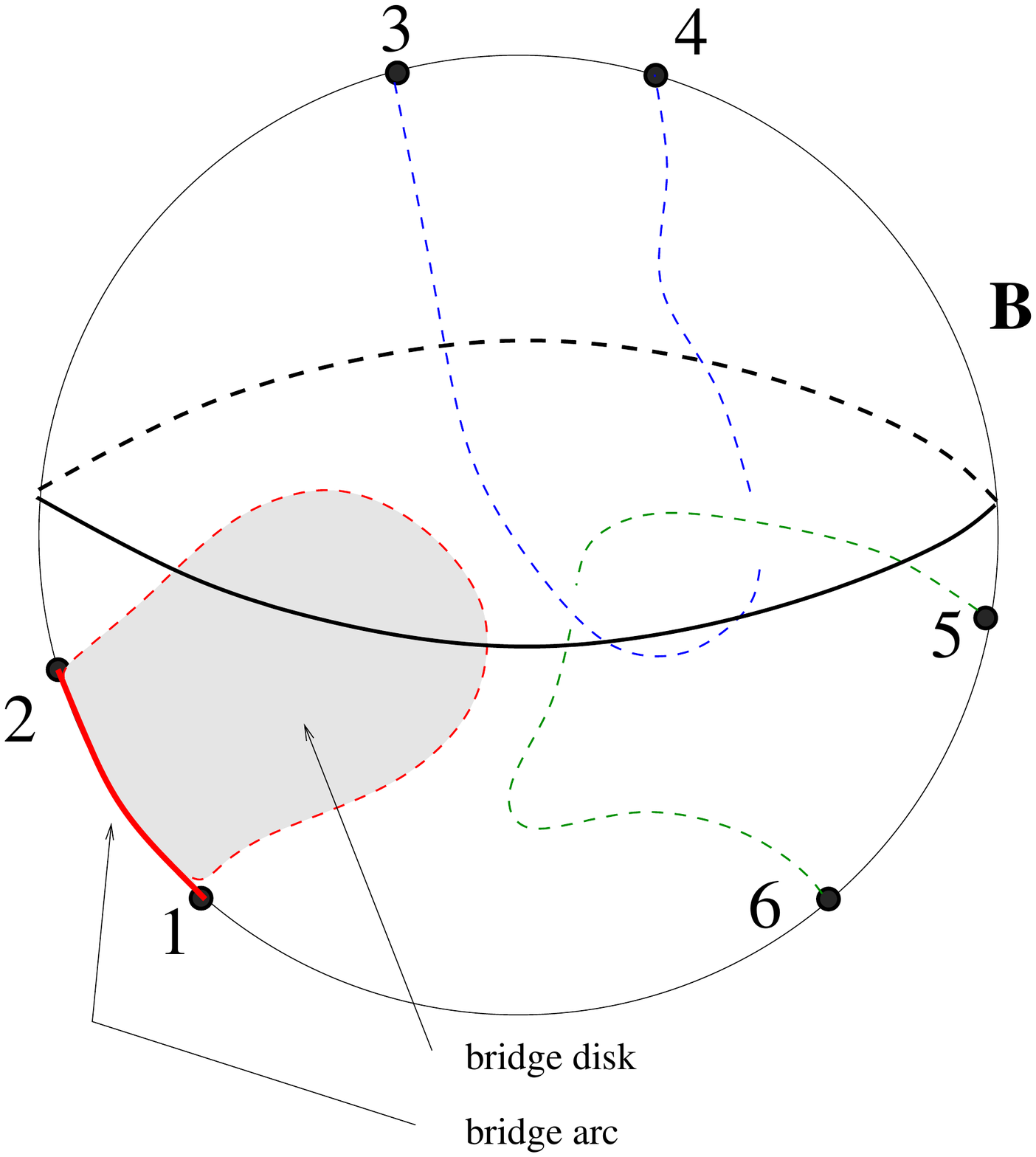}
\vskip -100pt
\caption{ Bridge disk and bridge arc in a rational $3$-tangle}
\label{F1}
\end{figure}

 Let $\tau=\tau_1\cup\tau_2\cup\tau_3$ be a rational $3$-tangle in $B$, where $\tau_1,\tau_2,\tau_3$ are the three strings of the rational $3$-tangle.  A  disk $D$ in $B$ is called a \textit{bridge disk} if $ $int $D\cap \tau=\emptyset$ and $\partial D=\tau_i\cup\beta$ and $\tau_i\cap \beta=\partial \tau_i=\partial \beta$ for some $i\in\{1,2,3\}$, where $\beta$ is a simple arc in $\partial B$ such that int $\beta\cap \tau=\emptyset$. The simple arc $\beta$ in $\Sigma_{0,6}$ is called a \textit{bridge arc} if  it cobounds a bridge disk with $\tau_i$ for some $i\in\{1,2,3\}$. We note that there is a collection of  disjoint three bridge arcs for a rational $3$-tangle since there are disjoint non-parallel three bridge disks in $B$ by definition of the rationality. However, we also note that the collection is not unique (up to isotopy) even for the same rational $3$-tangle. Once we select disjoint two bridge arcs on $\Sigma_{0,6}$, it is easy to find a third bridge arc. We note that any simple arc connecting the two remaining punctures which is disjoint with the existing two bridge arcs should be a bridge arc by Lemma~\ref{L1} below. So, there are infinitely many different bridge arcs for the same string of the rational $3$-tangle. We say that a simple closed curve $\gamma$ is \textit{obtained from} the bridge arc $\beta$ if the boundary of a small regular neighborhood of $\beta$ in $\Sigma_{0,6}$ is isotopic to $\gamma$.
 
\begin{Lem}\label{L1}
Let $\beta_1$ and $\beta_2$ be two disjoint  bridge arcs on $\Sigma_{0,6}$ of $(B,\tau)$, where $\tau$ is a rational $3$-tangle in $B$. Let $\beta_3$ be an arc connecting the two remaining punctures on $\Sigma_{0,6}$ which is disjoint with $\beta_1\cup\beta_2$. Then $\beta_3$ is also a bridge arc on  $\Sigma_{0,6}$ of $(B,\tau)$.
\end{Lem}

 Let $\mathcal{B}=\{\beta_1,\beta_2,\beta_3\}$ be a collection of  disjoint three bridge arcs on $\Sigma_{0,6}$ for $T=(B,\tau)$ which is called an $\textit{arc system}$ or simply a \textit{system} of $T$. For a better argument about this, we want to make a basic formation in $\Sigma_{0,6}$. Let $E_i$ be the fixed 2-punctured disks in $\Sigma_{0,6}$ as in the diagram of Figure~\ref{c2}. Let $E=E_1\cup E_2\cup E_3$, $\partial E=\partial E_1\cup\partial E_2\cup\partial E_3$ and $P=\Sigma_{0,6}\setminus E$. \\
 
 \begin{figure}[htb]
 \includegraphics[scale=.3]{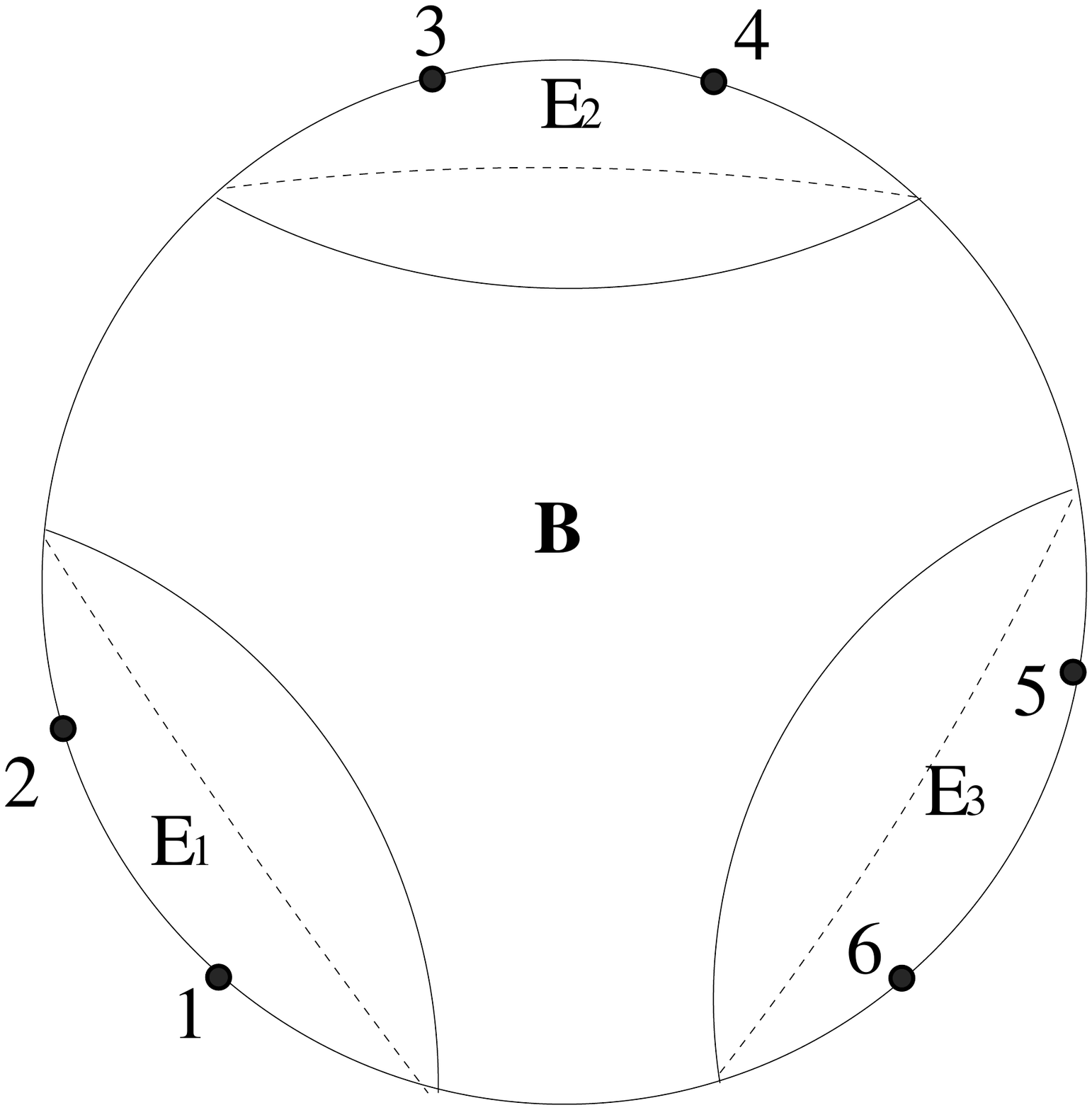}
 \vskip -80pt
 \caption{Basic formation}\label{c2}
 \end{figure}
 \subsection{Dehn's parameterization of an arc system of a rational $3$-tangle $T$}

Let $\mathcal{B}=\{\beta_1,\beta_2,\beta_3\}$ be an arc system of $T$. With a similar argument of \textit{Dehn's parameterization} of simple closed curves (Refer to~\cite{1}.), we can parameterize  $\beta=\beta_1\cup\beta_2\cup\beta_3$ as follows. Let $\omega_i$ be the subarc of $\partial E_i$ so that $\omega_i$ contains all of the intersection between $\partial E_i$ and $\beta$. They are called \textit{windows}. By considering the Dehn twists supported on $\partial E_i$, we can define the standard arcs in $P$ as in the diagrams of Figure~\ref{l10}. We note that all components of $P\cap \beta$ can be realized by giving the \textit{weights} of the standard arcs. The \textit{weight} of a standard arc stands for the number of parallel arcs to the  standard arc. 
\begin{figure}[htb]
\includegraphics[scale=.5]{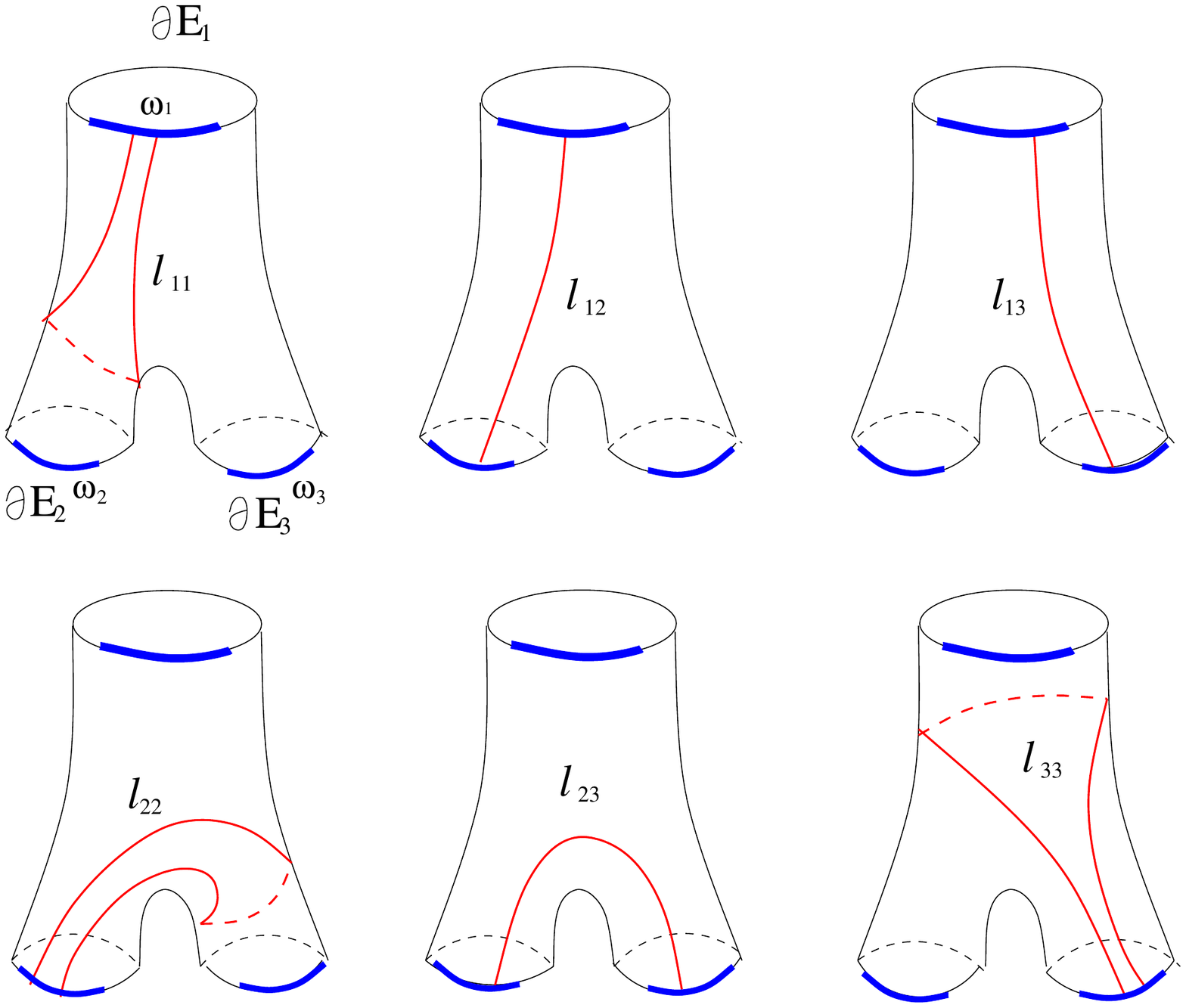}
\vskip -190pt
\caption{Standard arcs}\label{l10}
\end{figure}
Now, we investigate the arc types in $E_i$. For this, we define the innermost 2-punctured disk $E_i'$ so that  it contains only vertical intersections between $\beta$ and $E_i'$ as in the diagrams of Figure~\ref{c4}. We may need to isotope $\beta$ to have $E_i'$ so that it is in minimal general position with respect to $\partial E$ and $\partial E'$, where $\partial E'=\partial E_1'\cup\partial E_2'\cup\partial E_3'$.
 We may need to fix the shape of the half arcs in $E_i'$ as in the diagrams of Figure~\ref{c4}.
  
  \begin{figure}[htb]
\includegraphics[scale=.6]{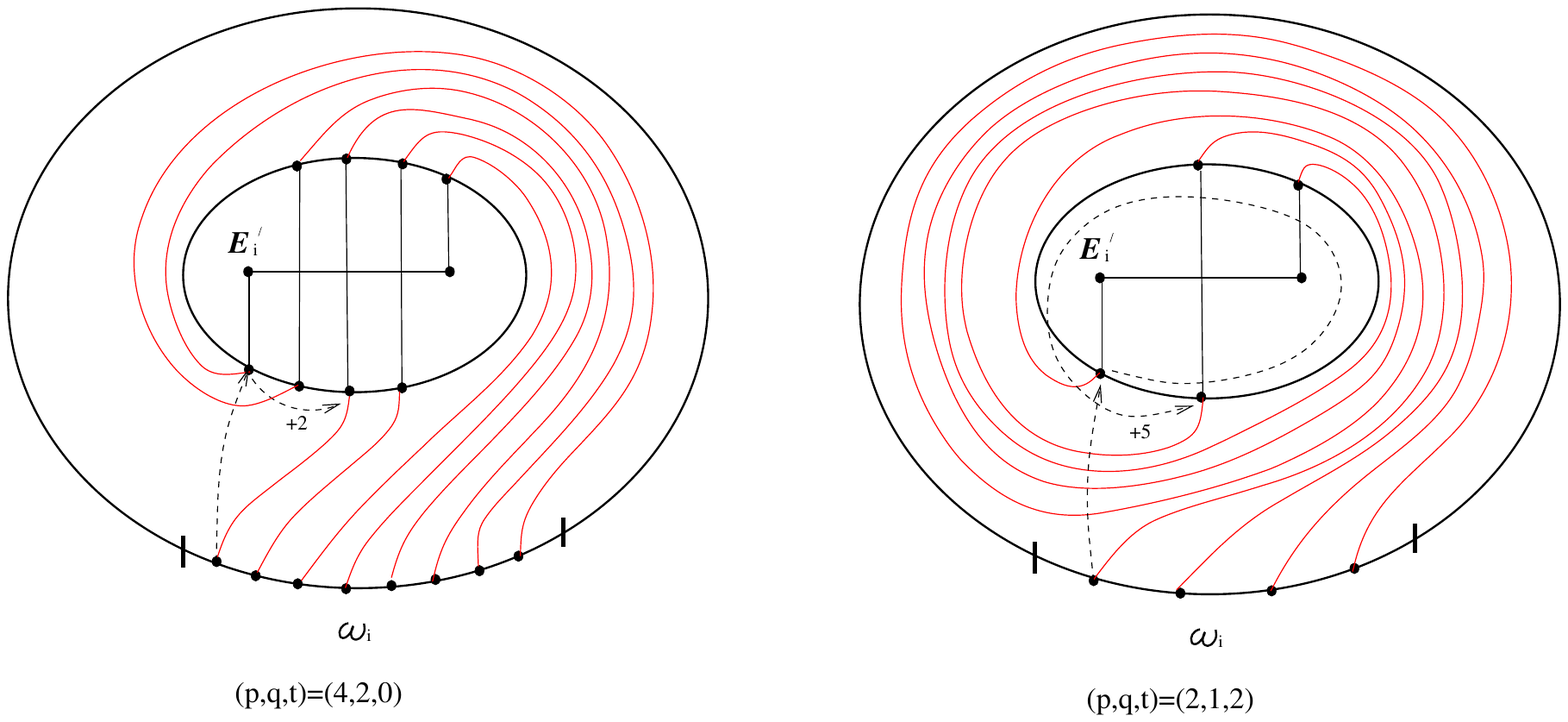}
\vskip -350pt
\caption{Arc types in $E_i$}\label{c4}
\end{figure}
We note that the two integers $p$ and $q$ determine  all the arc components of $\beta$ in $E_i$, where $p$ is the number of the arc components and $q$ gives us the connecting pattern between the $2p$ endpoints on $\omega_i$ and the endpoints on $\partial E_i'$. Refer to Figure~\ref{c4} to see the rule to determine $q$.  We note that one connecting pattern determines all the connecting patterns since  all the arc components should be disjoint each other. By adding up the all arguments above, we have the following theorem.

 \begin{Thm}[Special case II of Dehn's theorem]\label{T2} There is a one-to-one map $\phi :
 \mathcal{C}\rightarrow \mathbb{Z}^6$ so that $\phi([\beta_i])=(p_1,p_2,p_3, q_1, q_2, q_3)$, i.e., it classifies isotopy classes of  simple arcs (bridge arcs), where $\mathcal{C}$ is the collection of isotopy classes of simple arcs.
\end{Thm}

We note that  Theorem~\ref{T2} works for $\beta=\beta_1\cup\beta_2\cup\beta_3$ as well.
We say that an arc system of $T$ is in \textit{standard position} if every bridge arcs of the system are realized by taking the weights of the standard arcs and assigning integers for the connecting pattern in $E_i$. In other words, the bridge arcs in a system in standard position do not make a bigon with $\partial E$.
We want to point out that if $p_i=0$ then we are not able to decide a connecting pattern since there is no intersection between $\beta$ and $E_i$. In this case, we assign $0$ as $q_i$ for convenience.  
 We note that $\beta\cup \partial E$ is in minimal general position if $\beta$ is in standard position since the Dehn's parameterization is well defined.(Refer to \cite{1}.)

 \section{Normal forms of systems for a rational $3$-tangle $T$}

\subsection{Normal form and the bridge arc replacement}Let  $\mathcal{B}=\{\beta_1,\beta_2,\beta_3\}$ be an arc system of a rational $3$-tangle $T$. We assume that  $\beta=\beta_1\cup\beta_2\cup\beta_3$ is in standard position.  We say that an arc system $\mathcal{B}$ of $T$ is a \textit{normal form}  with respect to $\partial E$ if  there is no adjacent two intersections between $\beta$ and $\omega_i$ in $\omega_i$ which belong to the same $\beta_j$ for some $j\in\{1,2,3\}$.
 In order to show the existence of the normal form, we assume that there exist   more than or equal to two successive intersections not satisfying the definition of normal form as in Figure~\ref{F6}. We note that there are two subarcs of $\beta_j$ so that one end of them is one of the successive intersections   and the other end is one of the two endpoints of $\beta_j$ and they do not intersect  with the others among the successive intersections. Let $c_1$ and $c_2$ be the two subarcs of $\beta_j$. We define the two directions $(+)$ and $(-)$ as in the first diagram of Figure~\ref{F6}.  
Assume that $c_1$ and $c_2$ follow  $(+)$ and $(-)$ directions respectively as
 in the second diagram of Figure~\ref{F6}. Then we construct a new bridge arc with $c_1$ and $c_2$ so that it is disjoint from the other two bridge arcs not $\beta_j$ of the system as
 in the second diagram of Figure~\ref{F6}. Then, we take the bridge arc as one of the bridge arcs for a new arc system. It is clear that the new arc is also a bridge arc by Lemma~\ref{L1}.  If $c_1$ and $c_2$ follow only one of the  directions as in the third diagram, we construct a new bridge arc  with $c_1$ and $c_2$ so that  it is disjoint from the other two bridge arcs not $\beta_j$ of the system as in the fourth diagram of Figure~\ref{F6}. \begin{figure}[htb]
 \includegraphics[scale=.8]{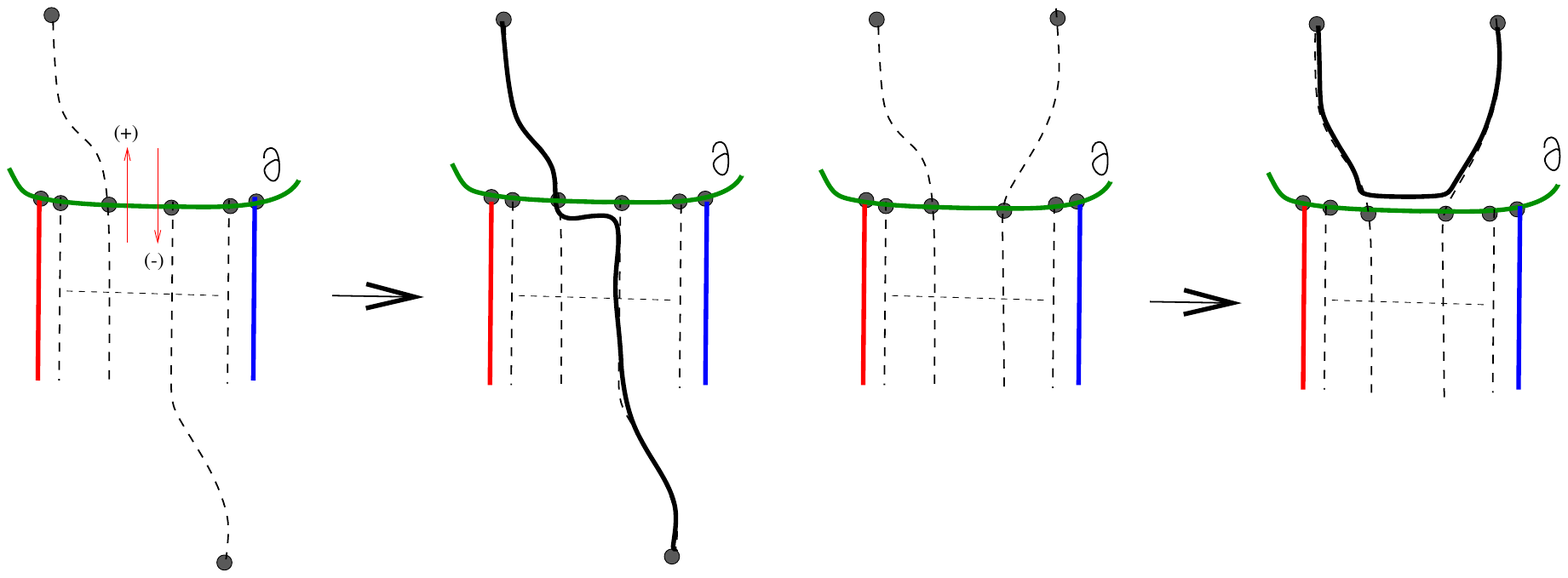} \vskip -500pt
  \caption{Bridge arc replacement, $
  \mathbf{BR}$}\label{F6}
 
 \end{figure}
 This procedure is called the \textit{bridge arc replacement}, briefly $\mathbf{BR}$ with respect to $\partial E$. We note that the new bridge arc obtained by $\textbf{BR}$ has less (geometric) intersections with $\partial  E$.

 \begin{Lem}
 Let $\mathcal{B}$ be an arc system of $T=(B,\tau)$. Then there exists a normal form $\mathcal{B}'$ of $T$ with respect to $\partial E$.
 \end{Lem}

 \begin{proof}
Since the number of intersections between $\partial P$ and $\beta$ is finite, we have a normal form by a sequence of \textbf{BR}s.
 \end{proof}

The first author \cite{2} proved the following theorem. The next corollary works by the theorem.

\begin{Thm}\label{T11}
Let $\mathcal{B}$ be a normal form of systems for the fixed trivial rational 3-tangle $T=(B,\epsilon)$ with respect to $\partial E$. Then $\mathcal{B}$ is unique up to isotopy. Especially, the three simple closed curves obtained from $\mathcal{B}$ respectively are isotopic to $\partial E_1,\partial E_2$ and $\partial E_3$ respectively.
\end{Thm}

\begin{Cor}\label{T12}
Let $\mathcal{B}=\{\beta_1,\beta_2,\beta_3\}$ be a system of $T=(B,\tau)$, where $\tau$ is an arbitrary rational $3$-tangle. Let $\gamma_1,\gamma_2,\gamma_3$ be the simple closed curves obtained from $\beta_1,\beta_2$ and $\beta_3$ respectively so that they are pairwise disjoint. Let $\gamma=\gamma_1\cup\gamma_2\cup\gamma_3$. Then there exists unique normal form of systems for $T$ with respect to $\gamma$ (not with respect to $\partial E$) up to isotopy. Moreover, it is $\mathcal{B}$.
\end{Cor}

\subsection{Normal form obtained from Normal jump moves} For a given arc system $\mathcal{B}=\{\beta_1,\beta_2,\beta_3\}$ for $T$,  we define a \textit{jump move} as follows. Take a regular neighborhood of $\beta_i$, named $N(\beta_i)$, in $\partial B$ which is disjoint with $\beta_j$ and $\beta_k$, where $\{i,j,k\}=\{1,2,3\}$.
Now, we consider a rectangle $R$ so that the interior of $R$ is disjoint with $N(\beta_i), \beta_j$ and $\beta_k$ and two parallel sides of $R$ are subarcs of $N(\beta_i)$ and $\beta_j$ respectively. Then we can construct a new bridge arc by a modified band sum as in the diagrams of Figure~\ref{c5}. Then the movement to have the new bridge arc  is called a \textit{jump move} (over $\beta_i$). We note that the bridge arc obtained by a bridge arc replacement also can be obtained by a jump move.

\begin{figure}[htb]
\includegraphics[scale=.5]{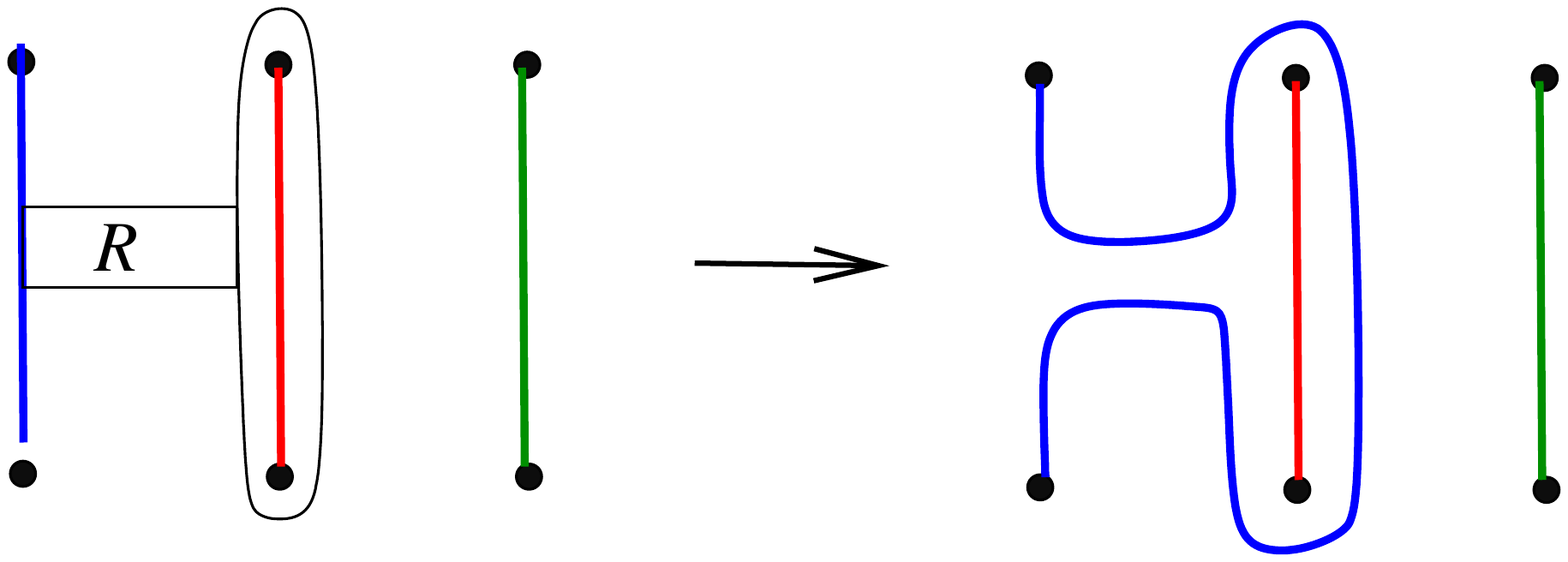}
\vskip -300pt
\caption{}\label{c5}
\end{figure}

Suppose that $\mathcal{B}$ is a normal form of systems for $T$ with respect to $\partial E$. If the new arc system $\mathcal{B}'$  replacing one of the bridge arcs by a jump move is also a normal form then the jump move is called a \textit{normal jump move}.
Now, we discuss how to find  a new normal form $\{\beta_1',\beta_2,\beta_3\}$ from the normal form $\mathcal{B}=\{\beta_1,\beta_2,\beta_3\}$ for $T$ with respect to  $\partial E$ by using a normal jump move.\\

 \textbf{Case 1}: $E_i$ contains one of the bridge arcs of the normal form; the following theorem gives an information to know the given rational $3$-tangle $T$.

\begin{Thm}
Suppose that $\mathcal{B}=\{\beta_1,\beta_2,\beta_3\}$ and $\mathcal{B}'=\{\beta_1',\beta_2',\beta_3'\}$ are  normal forms for the two rational $3$-tangles $T$ and $T'$ respectively with respect to $\partial E$. Let $(0,0,p_2,q_2,p_3, q_3)$ and $(0,0,p_2',q_2',p_3', q_3')$ be the ordered sequences of Dehn's parameters  of $\mathcal{B}$ and $\mathcal{B}'$ respectively.
 Then $T$ and $T'$ are isotopic if and only if $p_i=p_i'$ for $i=2,3$ and $q_2-q_3=q_2'-q_3'$.
\end{Thm}
\begin{figure}[htb]
\includegraphics[scale=.8]{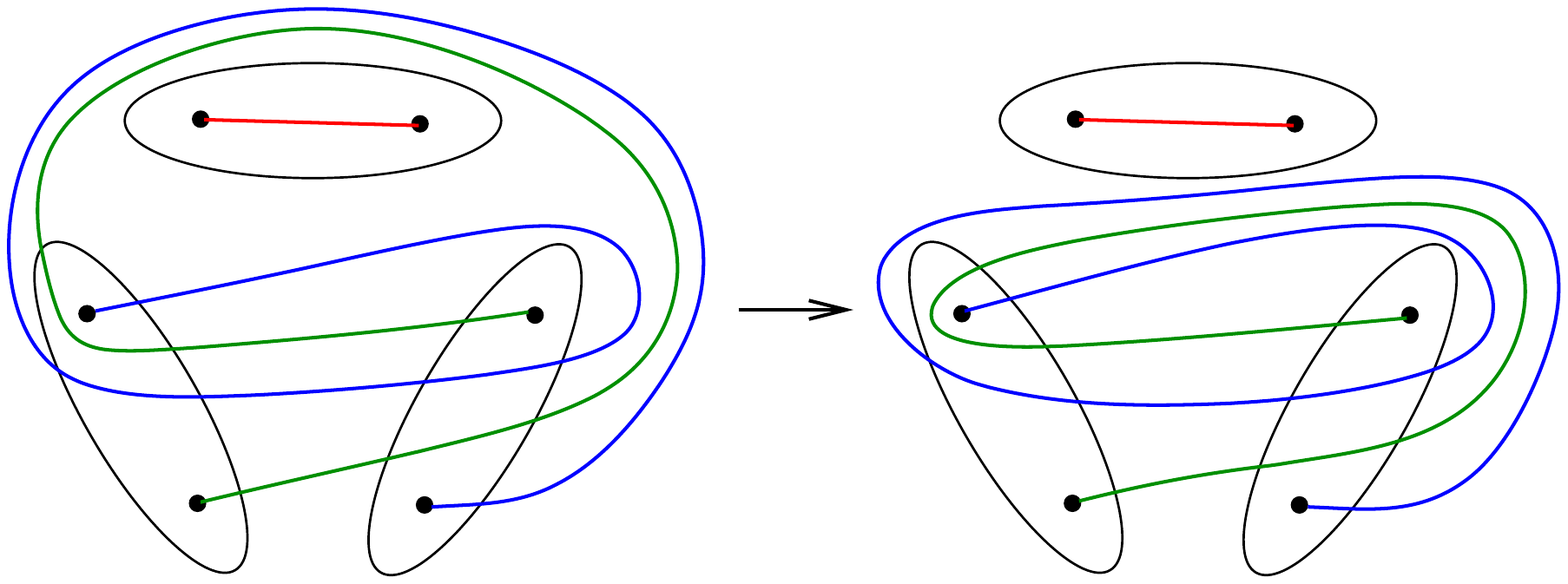}
\vskip -520pt
\caption{}\label{c10}
\end{figure}

\begin{proof}

Let $\beta=\beta_1\cup\beta_2\cup\beta_3$ and  $\beta'=\beta_1'\cup\beta_2'\cup\beta_3'$. Since $p_1=p_1'=0$,  $E_1$ contains one of the bridge arcs of $\mathcal{B}$ and $\mathcal{B}'$. Let $\beta_1$ and $\beta_1'$ be the bridge arcs of each normal form. We note that they are isotopic.
 Both of $\mathcal{B}$ and $\mathcal{B}'$ can have new normal forms obtained from  sequence of jump moves over $\beta_1$ as in the diagrams of Figure~\ref{c10}.  We note that there is no change of $p_i$ and $p_i'$ for $i=2,3$ after the sequence of  jump moves.
The right diagram of Figure~\ref{c10} has two parts; $\beta_1$ and the  two bridge arcs representing a rational $2$-tangles. First, we assume that $T$ and $T'$ are isotopic. Then, the rational $2$-tangles represented by the pairs of two bridge arcs are isotopic. We note that a  rational $2$-tangle is determined by $p$ and $q$ uniquely, where $p$ is the minimal intersection number between the two bridge arcs and $E_2$(or $E_3$) and $q$ represents the connecting pattern of endpoints between $\omega_2$ and $\omega_3$. Actually, we note that $q=q_2-q_3=q_2'-q_3'$. For the opposite direction, we note that if $q_2-q_3=q_2'-q_3'$ then the two corresponding rational $2$-tangles are isotopic. Therefore, $T$ and $T'$ are isotopic. This completes the proof.
 
\end{proof}

\textbf{Case 2}: A rational 3-tangle $T$ has a normal form $\mathcal{B}=\{\beta_1,\beta_2,\beta_3\}$ with respect to $\partial E$ satisfying the condition that $|\beta\cap \partial E_i|\geq 2$ for all $i=1,2,3$, where $\beta=\beta_1\cup\beta_2\cup\beta_3$.\\

 In this case, we use the normal jump moves to find an optimized normal form. We first define the \textit{standard normal jump move}. We show that the standard normal jump move is the only normal jump move in Theorem~\ref{L5}.\\
  \begin{figure}[htb]
 \includegraphics[scale=.9]{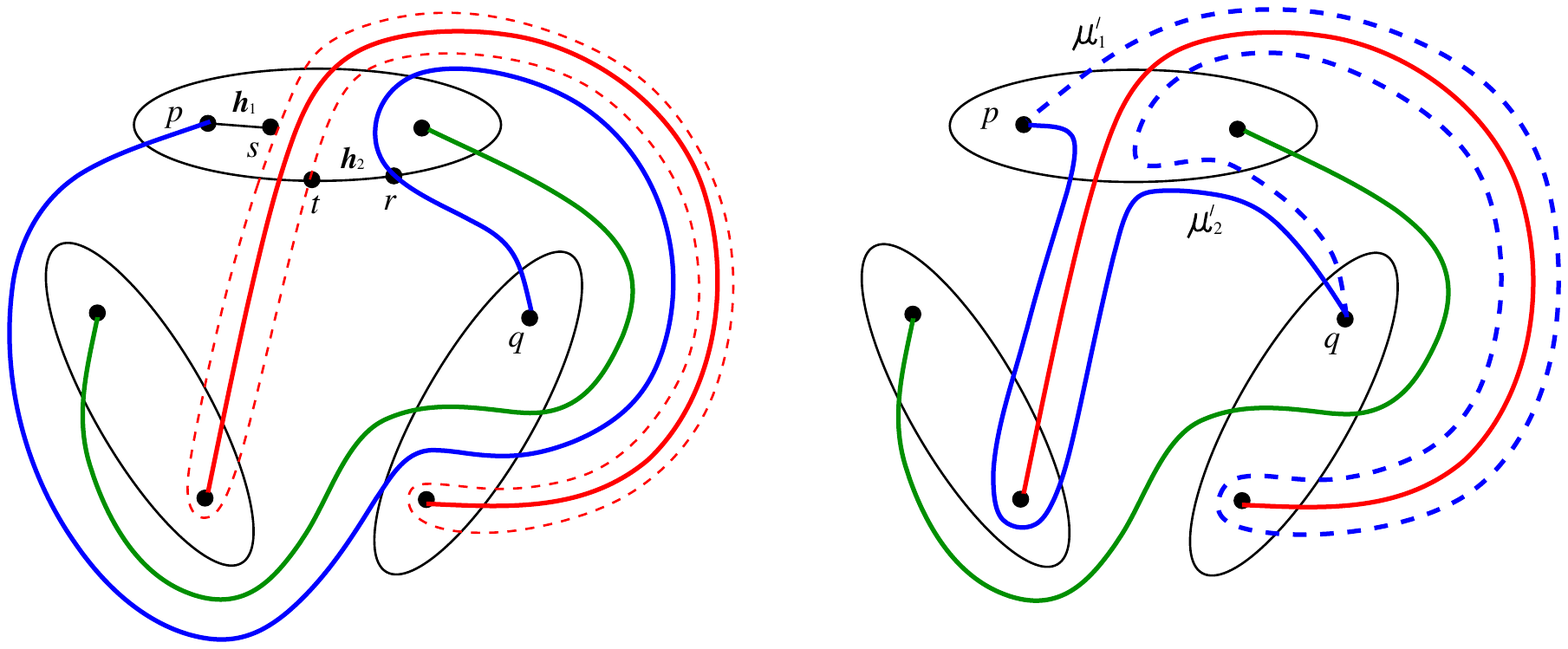}
 \vskip -550pt
 \caption{}\label{c14}
 \end{figure}
 
$\circ$ \textit{Standard normal jump move} : Let $p$ be one of the endpoints of $\beta_1$ which is in $E_i$. Then either $\beta_2$ or $\beta_3$ is adjacent to $p$ in $E_i$ since $\mathcal{B}$ is a normal form and $|\beta\cap\partial E_i|\geq 2$. Without loss of generality, we assume that $\beta_2$ is the  bridge arc adjacent to $p$.
Now, we take a small regular neighbhorhood of $\beta_2$, $N(\beta_2)$, so that it does not intersect with $\beta_1$ and $\beta_3$  as in the  first diagram of Figure~\ref{c14}. We now take a shortest path $h_1$ from $p$ to a point $s$ of $\partial N(\beta_2)$ in $E_i$ as in the first diagram of Figure~\ref{c14}. Since $\beta_2$ is the adjacent bridge arc to $p$, $h_1$ does not meet with $\beta$. We note that $\beta_1$ has at least one  intersection  with $\partial E$ so that it has adjacent intersection in $\partial E$ which belongs to $\beta_2$ since $\beta_2$ is the adjacent bridge arc to $p$. Let $r$ be the first intersection of them when we follow $\beta_1$ from $q$. Let $t$ be an  intersection between $\partial N(\beta_2)$ and $\partial E$ so that the subarc of $\partial E$ connecting $r$ and $t$ does not intersect $\beta_2\cup\beta_3$. Let $h_2$ be the subarc of $\partial E$. Let $\widetilde{h_2}$ be the path from $q$ to $t$ which is the union of the subarc of $\beta_1$ between $q$ and $r$, and $h_2$.
  There are two  arc components $\beta_{21}$ and $\beta_{22}$ of $\partial N(\beta_2)\setminus\{s,t\}$ from $s$ to $t$.  Let $\mu_1=h_1\cup\beta_{21}\cup\widetilde{h_2}$ and  $\mu_2=h_1\cup\beta_{22}\cup\widetilde{h_2}$. We note that $\mu_1$ and $\mu_2$ are disjoint with $\beta_2\cup\beta_3$. Moreover, we can isotope $\mu_1,\mu_2$ so that they are disjoint except the two endpoints and they meet $\beta_1$ only at the endpoints. Let $\mu_1'$ and $\mu_2'$ be the isotoped simple arc from $\mu_1$ and $\mu_2$ respectively. Since $\mu_1'\cup\mu_2'$ separates $\beta_2$ and $\beta_3$, $\beta_1$ should be isotopic to one of $\mu_i'$. Without loss of generality, suppose that $\mu_1'$ is isotopic to $\beta_1$. We note that $\mu_2$ is obtained from $\mu_1$ by a jump move over $\beta_2$.\\
  
This jump move is a normal jump move called \textit{standard normal jump move} of $\beta_1$. 
 We isotope $\mu_2'$ so that it  is in minimal general position with respect to $\partial E$.  We claim that it is a normal jump move.\\
 
\begin{proof}[Proof of the claim] :  Take an intersection $a$ between $\beta_1$ and $\partial E_j$. Since $\{\beta_1,\beta_2,\beta_3\}$ is a normal form with respect to $\partial E$, there are three cases for the two adjacent intersections to $a$ as in the diagram $(a-1), (a-2)$ and $(a-3)$ of Figure~\ref{c6}. In $(a-1)$, there are two dotted red lines which are adjacent to $a$. We note that one of them would be disappeared when we take $\mu_2'$ since one of them is a part of $\mu_1'$ which is isotopic to $\beta_1$. In $(a-2)$, there is no subarc of $\mu_2'$ between the green and red arcs since the red dotted arc between them should be a part of $\mu_1'$ which is isotopic to $\beta_1$. In $(a-3)$, if the blue arc between the green arcs is a part of $\widetilde{h_2}$ then we need to consider the two dotted blue arcs which are parts of $\mu_1'$ and $\mu_2'$ respectively. Clearly, only one of the dotted blue arcs is a part of $\mu_2'$. Now, consider the case that the green and the red arcs are adjacent as in the diagram $(b)$ of Figure~\ref{c6}. It does not matter where or not the dotted red arc is a part of $\mu_2'$ in this case to check the normality, but it should be. Finally, we conclude that $\{\mu_2',\beta_2,\beta_3\}$ is a normal form, so it is a normal jump move.

  \begin{figure}[htb]
 \includegraphics[scale=.8]{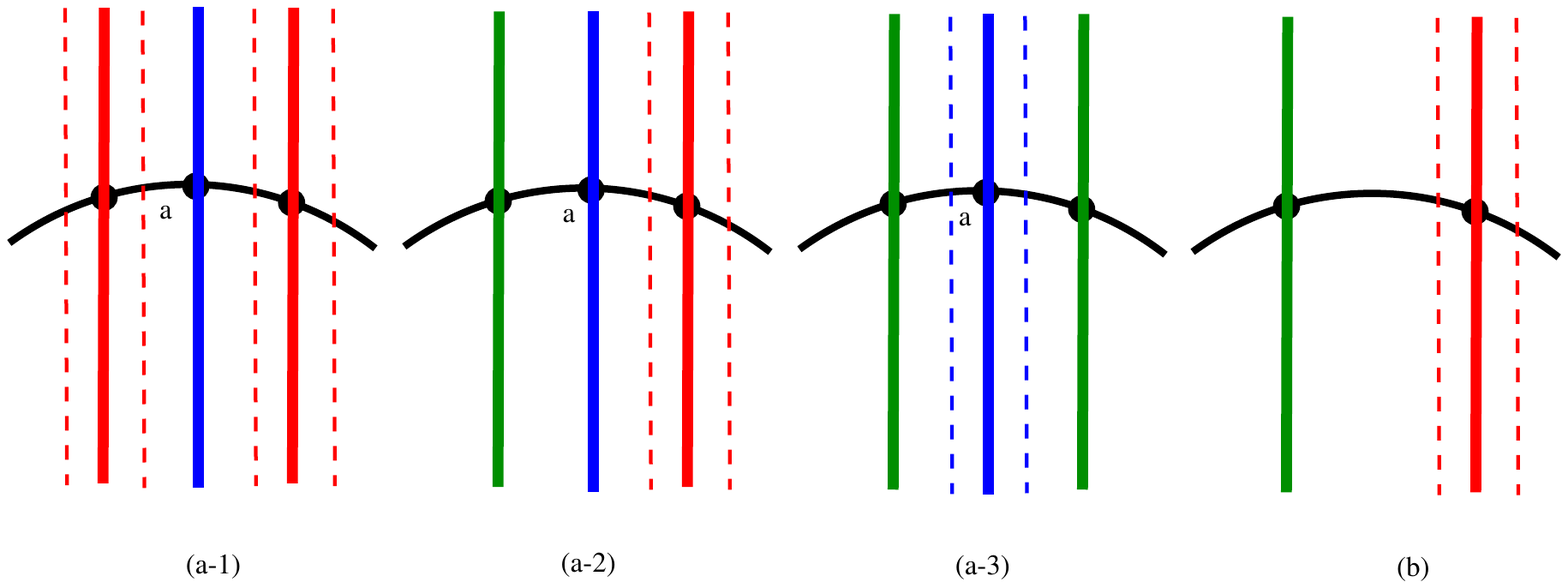}
 \vskip -520pt
 \caption{}\label{c6}
 \end{figure} 
 
 \end{proof}
 
  The following theorem shows the uniqueness of the normal jump move.

\begin{Thm}\label{L5}
Let $\mathcal{B}=\{\beta_1,\beta_2,\beta_3\}$ be a normal form of $T$ with respect to $\partial E$.
 Then there exists unique bridge arc $\beta_1'$ replacing $\beta_1$ so that it is not isotopic to $\beta_1$  and $\{\beta_1',\beta_2,\beta_3\}$ is a normal form of $T$ with respect to $\partial E$.
\end{Thm}

\begin{proof}

 Let $\beta_1'=\mu_2'$. Then, we note that $\beta_1\cup\beta_1'$  separates $\beta_2$ and $\beta_3$. The first diagram of Figure~\ref{c9}  is the schematic picture to explain the situation.   
Let $A$ and $B$ be the two disk regions bounded by $\beta_1\cup\beta_1'$ in $S^2$. Assume that $A$ contains $\beta_2$ and $B$ contains $\beta_3$. 
\begin{figure}[htb]
\includegraphics[scale=.5]{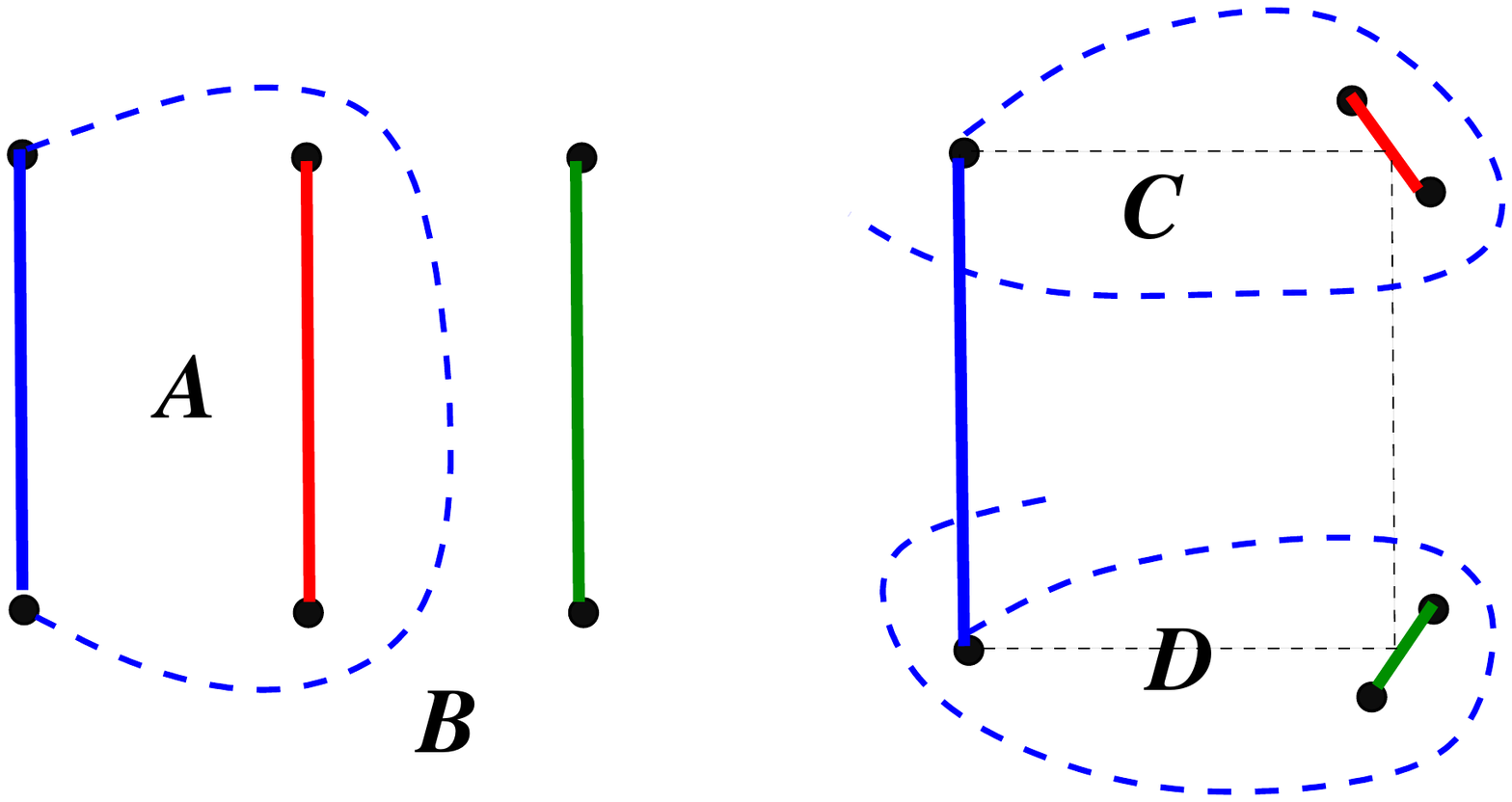}
\vskip -280pt
\caption{}\label{c9}
\end{figure}
In order to have a contradiction, we assume that there exists $\beta_1''$ which is  isotopic to none of $
\beta_1$ and $\beta_1'$ and $\{\beta_1'',\beta_2,\beta_3\}$ is a normal form with respect to $\partial E$.
We note that $\beta_1\cup\beta_1''$ separates $\beta_2$ and $\beta_3$ as well. We first assume that $\beta_1\cup\beta_1''$ separates $S^2$ into two connected components as in the first diagram of Figure~\ref{c9}. We consider the point components of $(\beta_2\cup\beta_3)\cap\partial E$. We first assume that there are adjacent two different coloured points of $(\beta_2\cup\beta_3)\cap\partial E$ in $\partial E_i$ for some $i$ and $\beta_1$ passes through the subarc of $\partial E_i$ connecting the two points. Then $\beta_1''$ cannot pass through the same subarc since a simple path from a point of $\beta_2$ should be back to the component $A$ when it meets $\beta_1\cup\beta_1''$ twice. So, the simple path cannot end at a point of $\beta_3$.\\

 Secondly, we assume that there are adjacent two different coloured components of  $(\beta_2\cup\beta_3)\cap\partial E$ in $\partial E_i$ for some $i$ and $\beta_1$ do not pass through the subarc of $\partial E_i$ connecting the two components. We claim that $\beta_1''$ should pass through the subarc of $\partial E_i$.  It is followed by the fact that a path from a point of $\beta_2$ to a point of $\beta_3$ should meet $\beta_1\cup\beta_1''$.\\
 
 At last, we assume that there are adjacent two components of  $(\beta_2\cup\beta_3)\cap\partial E$ with same color  in $\partial E_i$ for some $i$. It is clear that $\beta_1$  passes through the subarc of $\partial E_i$ connecting them because of the normality. We note that $\beta_1''$ should pass through the same subarc of $\partial E_i$ because of the normality.\\

 From the three cases above, we note that the intersecting patterns of $\beta_1'$ and $\beta_1''$ with  $\partial E$ are the same. This implies that $\beta_1'$ and $\beta_1''$ are isotopic by Theorem~\ref{T2}.\\

Now, we assume that $\beta_1\cup\beta_1''$ separates $S^2$ into more than two connected components. Let $C$ and $D$ be the connected regions containing $\beta_2$ and $\beta_3$ respectively as in the second diagram of Figure~\ref{c9}. We note there is no common face between $C$ and $D$. In other words, they meet at a point or they are disjoint. By using a similar argument for the previous case, we can check that the intersecting patterns of $\beta_1$ and $\beta_1''$ with $\partial E$ are the same. This implies that $\beta_1''$ is isotopic to $\beta_1$. This makes a contradiction and it completes the proof.
\end{proof}

\begin{Thm}\label{L7}
Let $\mathcal{B}$ and $\mathcal{B}'$ be two normal forms of $T$ with respect to $\partial E$. Then $\mathcal{B}$ can be obtained from $\mathcal{B}'$ by a sequence of  normal jump moves. 
\end{Thm}
\begin{proof}
\begin{figure}[htb]
\includegraphics[scale=.88]{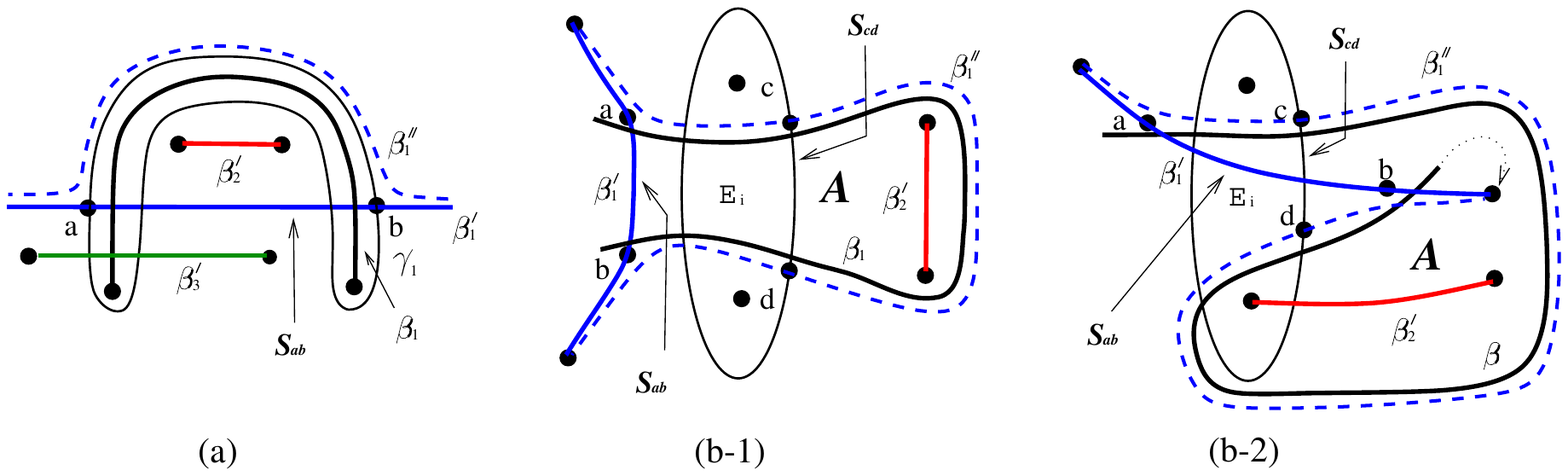}
\vskip -600pt
\caption{}\label{c15}
\end{figure}

Let $\mathcal{B}=\{\beta_1,\beta_2,\beta_3\}$ and $\mathcal{B}'=\{\beta_1',\beta_2',\beta_3'\}$ be  two distinct normal forms of $T$ with respect to $\partial E$. Let $\gamma_1,\gamma_2,\gamma_3$ be simple closed curves obtained from $\beta_1,\beta_2,\beta_3$ respectively so that they are pairwise disjoint. Let $\gamma=\gamma_1\cup\gamma_2\cup\gamma_3$. 
By Corollary~\ref{T12}, there exist $\gamma_i$ and $\beta_j'$ so that two of the intersections between $\gamma_i$ and $\beta_j'$ are adjacent in $\gamma_i$.   Without loss of generality,  we assume that $\beta_1'$ has a pair of adjacent intersections $a$ and $b$ with $\gamma_1(\subset\gamma)$. In other words, the subarc of $\gamma_1$ between $a$ and $b$  have no intersection with $\beta'=\beta_1'\cup\beta_2'\cup\beta_3'$ as in the diagram (a) of Figure~\ref{c15}. Let $S_{ab}$ be the subarc of $\beta_1'$ between $a$ and $b$. We emphasize that the diagram (a) is one of possible schematic diagrams. Let $\beta_1''$ be the bridge arc obtained from $\beta_1'$  by the \textbf{BR} with respect to $\gamma$ as in the  diagram (a) of Figure~\ref{c15}. We isotope $\beta_1''$ slightly so that it does not intersect with $\beta_1'$ except the two  endpoints. For convenience, we use the modified simple arc as $\beta_1''$. We note that $\beta_1''$ is not isotopic to $\beta_1'$ since $|\beta_1''\cap \gamma|<|\beta_1'\cap \gamma|$. Moreover, it does not intersect with $\beta_2'\cup\beta_3'$.
So, we note that $\beta_1'\cup\beta_1''$ is a simple closed curve which separates $\beta_2'$ and $\beta_3'$. \\

Now, we show that $\{\beta_1'',\beta_2',\beta_3'\}$ is a normal form of $T$ with respect to $\partial E$. Then this completes the proof of this theorem since $|(\beta_1''\cup\beta_2'\cup\beta_3')\cap\gamma|<|(\beta_1'\cup\beta_2'\cup\beta_3')\cap\gamma|<\infty$. \\

We first claim that there is no adjacent  intersections belonging to only $\beta_2'$ or $\beta_3'$ in $\partial E$  when we have $\{\beta_1'',\beta_2',\beta_3'\}$. 
If there exist adjacent intersections belonging to only $\beta_2'$ or $\beta_3'$, then there was a single intersection of $\beta_1'$ with $\partial E$ between them since $\{\beta_1',\beta_2',\beta_3'\}$ is a normal form of $T$ with respect to $\partial E$. Since $\beta_1'\cup\beta_1''$ separates $\beta_2'$ and $\beta_3'$, the given intersections should belong to $\beta_2'$ and $\beta_3'$ respectively. This makes a contradiction. Therefore, the claim works.\\

Now, we claim that there is no  adjacent  intersections belonging to only $\beta_1''$ in $\partial E$ as well when we have $\{\beta_1'',\beta_2',\beta_3'\}$.  In order to prove the claim, we assume that there are such adjacent intersections. Then, there are two possible cases as  the diagrams (b-1) and (b-2) of Figure~\ref{c15}, where the points $c$ and $d$ are the two adjacent intersections belonging to $\beta_1''$ in $\partial E$.  Let $S_{cd}$ be the subarc of $\partial E_j$ which connects $c$ and $d$ so that it contains no more intersections with $\beta_1''\cup\beta_2'\cup\beta_3'$. Then $S_{cd}$ and the subarc of $\beta_1''$ cobounds a simple closed curve $S$ so that $S$ encloses one of $\beta_2'$ and $\beta_3'$. Without loss of generality, we assume that $S$ encloses $\beta_2'$. Since $\beta_1'\cup\beta_1''$ separates $\beta_2'$ and $\beta_3'$ and $S$ separates $\beta_2'$ and $\beta_3'$, the intersection patterns between $\beta_2\cup\beta_3$ and $S_{ab}$, and between  $\beta_2\cup\beta_3$ and $S_{cd}$ are the same up to isotopy.   If either $\beta_2$ or $\beta_3$ is isotopic to $\beta_2'$ then it violates the condition that $\mathcal{B}$ is a normal form with respect to $\partial E$ since either $(\beta_2\cup\beta_3)\cap S_{cd}$ is empty set or there are more than one intersection of $\beta_2$ or $\beta_3$ with $S_{cd}$. 
So, we assume that none of $\beta_2$ and $\beta_3$ is isotopic to $\beta_2'$. Let $A$ be the region bounded by the simple closed curve $\widehat{S}$ consisting of the subarcs of $S_{ab}$(or the extended arc of $S_{ab}$) and  $\beta_1$ which contains $\beta_2'$ as in the diagrams of Figure~\ref{c15}. We note that the interior of $A$ does not intersect with $\beta_1$.
 We also note that there exists a simple arc connecting the remaining two punctures not the four punctures in $\beta_1\cup\beta_2'$ so that it is disjoint with  $\beta_1\cup\beta_2'$ and $A$. Then it would be a bridge arc by Lemma~\ref{L1}. Let $\beta_3''$ be the bridge arc. Then $\{\beta_1,\beta_2',\beta_3''\}$ is a system of the same rational $3$-tangle $T$.   By taking pairwise disjoint three simple closed curves $\gamma_1,\gamma_2'$ and $\gamma_3''$ obtained from $\beta_1,\beta_2'$ and $\beta_3''$, we investigate $\gamma_2$ and $\gamma_3$. We note that $\{\beta_1,\beta_2,\beta_3\}$ is also a system of $T$. By Corollary~\ref{T12}, we note that there exist adjacent two intersections of $\beta_2\cup\beta_3$ with $S_{ab}$  belonging to the same $\beta_2$ or $\beta_3$ since $\beta_1$ is disjoint with $\gamma_1\cup\gamma_2'\cup\gamma_3''$. Moreover, we recall that  the intersection patterns between $\beta_2\cup\beta_3$ and $S_{ab}$, and $\beta_2\cup\beta_3$ and $S_{cd}$ are the same. It makes a contradiction to the condition that $\mathcal{B}=\{\beta_1,\beta_2,\beta_3\}$ is a normal form with respect to $\partial E$. This completes the claim. Finally, we proved that $\{\beta_1'',\beta_2',\beta_3'\}$ is a normal form of $T$ with respect to $\partial E$.\\

\end{proof}

 \section*{Acknowledgements}

 This research was supported by Basic Science Research Program through the National Research Foundation of Korea (NRF)  funded by the the Ministry of Education (NRF-2020R1I1A1A01052279).

\end{document}